%% file: main.tex
\documentclass{article}
\input{preamble.tex}

\begin{document}
\title{On Whitehead's cut vertex lemma}
\author{Rylee Alanza Lyman}
\maketitle

\begin{abstract}
    One version of Whitehead's famous cut vertex lemma says that
    if an element of a free group is part of a free basis,
    then a certain graph associated to its conjugacy class that we call the star graph
    is either disconnected or has a cut vertex.
    We state and prove a version of this lemma
    for conjugacy classes of elements and convex-cocompact subgroups
    of groups acting cocompactly on trees with finitely generated edge stabilizers.
\end{abstract}

\input{introduction.tex}
\input{themaintheorem.tex}

\bibliographystyle{alpha}
\bibliography{bib.bib}
\end{document}

%% file: preamble.tex
\usepackage[margin=1.4in]{geometry}
\usepackage{amsfonts} 
\usepackage{amsthm} 
\usepackage{amsmath} 
\usepackage{tikz-cd} 
\usepackage{import} 
\usepackage{hyperref} 
\usepackage{cleveref} 

\newtheorem{THM}{Theorem}
\newtheorem{prop}{Proposition}
\newtheorem{lem}[prop]{Lemma}
\theoremstyle{definition}
\newtheorem{ex}[prop]{Example}

\newcommand{\doublebackslash}{\backslash\mkern-5mu\backslash}
\DeclareMathOperator{\st}{st}

%% file: introduction.tex
Given a collection of conjugacy classes $C$ of elements
of a finite-rank free group,
expressed as cyclically reduced words in some free basis,
one associates a \emph{star graph,} 
which records the adjacency of letters of the free basis in cyclic words in $C$.
A famous lemma of Whitehead \cite{Whitehead} implies that if
$C = \{x\}$, where $x$ is the conjugacy class of an element of some free basis,
then the star graph of $C$ is either disconnected or has a cut vertex.

Recently, Bestvina--Feighn--Handel \cite{BestvinaFeighnHandel}
gave an expanded definition of the star graph of a
collection of conjugacy classes of finitely generated subgroups of free groups.

The purpose of this note is to prove a cut vertex lemma
for certain conjugacy classes of elements and convex-cocompact subgroups 
of groups acting cocompactly on trees with finitely generated edge stabilizers,
or equivalently, of fundamental groups of finite graphs of groups with finitely generated edge groups.

To wit, let $G$ be a group acting on a tree $T$ with finitely generated edge stabilizers.
We say $T$ and a $G$-tree $S$ belong to the same \emph{reduced deformation space}
in the sense of Forester and Guirardel--Levitt \cite{Forester,GuirardelLevittDeformation}
if there are $G$-equivariant continuous maps $T \to S$ and $S \to T$
which become simplicial maps (edges to edges, vertices to vertices)
after $G$-equivariantly subdividing edges in the domain tree.
(Guirardel--Levitt call such maps ``morphisms.'')
A tree $S'$ is a \emph{collapse} of $S$ if the $G$-tree $S'$ may be obtained from $S$
by equivariantly collapsing edges of $S$ to vertices.
We say that $T$ \emph{strongly dominates} any tree $S'$ which is a collapse 
of a tree in the same reduced deformation space as $T$.
A subgroup $H$ of $G$ is \emph{convex-cocompact} in $T$
if $H$ acts on the $H$-minimal subtree of $T$ cocompactly.
Given such an $H$ or an element $g \in G$,
we say $H$ and $g$ are \emph{simple} if there exists a tree $S'$ strongly dominated by $T$
with one orbit of edges with the property that $H$ (or respectively $g$) is elliptic in $S'$.
A collection $C$ of conjugacy classes of elements and convex-cocompact subgroups of $G$
is \emph{jointly simple} if there is some tree $S'$
strongly dominated by $T$
in which all conjugacy classes in $C$ are elliptic.

\begin{THM}
    \label{maintheorem}
    If a collection $C$ of conjugacy classes of elements and convex-cocompact subgroups of $G$ 
    is jointly simple, then the star graph of $C$
    is either disconnected or has a cut vertex.
\end{THM}

For example, if $F_n$ is nonabelian free and acts freely and cocompactly on $T$,
then the reduced deformation space of $T$ is the 
\emph{unreduced} Culler--Vogtmann Outer Space \cite{CullerVogtmann}
(separating edges are allowed)
and any tree $S'$ strongly dominated by $T$ is a one-edge free splitting of $G$,
to which we may associate a free factorization of $F_n$
of the form $F * \mathbb{Z}$ or $F_1 * F_2$.
An element or subgroup is elliptic in $S'$ if and only if it is conjugate
either into $F$ in the first case, or into $F_1$ or $F_2$ in the second case.
Thus in this situation an element $g \in F_n$ is \emph{simple} in our sense if and only if
it is \emph{simple} in the sense that it is contained in some proper free factor of $F_n$.

Our method of proof of \Cref{maintheorem} is inspired 
by a paper of Heusener and Weidmann \cite{HeusenerWeidmann}.
In that paper, they consider conjugacy classes of elements in a finite-rank free group $F_n$,
so $F_n\backslash T$ is a graph with one vertex, a \emph{rose.}
Any tree in the reduced deformation space of $T$
yields a graph $G'$ with a map $G' \to G$
that may be written as a finite sequence of Stallings folds.
They consider a special class of graphs they term \emph{almost roses} for which a single fold suffices.
They show that if $C$ is a jointly simple collection of conjugacy classes (their word is \emph{separable}),
then the star graph of $C$ is a subgraph of the star graph of some almost rose,
and that star graphs of almost roses have cut vertices.

Like Heusener and Weidmann, we work in the quotient of the tree, which is a graph of groups.
Using \cite{BestvinaFeighnGTrees}, we also define a class of graphs of groups
which we term \emph{almost $\mathbb{G}$} which fold onto our fixed reduced graph of groups $\mathbb{G}$
with a single fold.
Our definition of the star graph follows \cite{BestvinaFeighnHandel} 
and is thus slightly different from Heusener and Weidmann's.
It is not literally true that the star graph of $C$ in our sense is a subgraph of the star graph
of an almost $\mathbb{G}$ graph of groups,
but there is still a good relationship, see \Cref{mapsgivemaps}.
We show in \Cref{cutvertex} that our star graph of an almost $\mathbb{G}$ graph of groups has a cut vertex
in a particular sense: in general $\mathbb{G}$ has more than one vertex
and thus the star graph always has multiple components;
we say that the star graph is connected 
if it has exactly one connected component for each vertex of $\mathbb{G}$
and that it has a cut vertex if one of these components has a cut vertex.
The theory of Stallings folds completes the proof.

%% file: themaintheorem.tex
\section{Stallings graphs of groups}

Suppose $G$ acts cocompactly on a tree $T$.
Associated to this action there is a quotient \emph{graph of groups} $\mathbb{G} = G\doublebackslash T$.
We assume a certain familiarity with graphs of groups in this note,
and refer the reader interested in more information to \cite{Trees,Bass,ScottWall,Myself}.
We follow the notation in \cite[Section 1]{Myself}.
We fix $\mathbb{G}$ throughout this note
and suppose that it is \emph{reduced} in the sense of \cite{Forester}:
in $\mathbb{G}$ this means that if an edge $e$ has 
a surjective edge-to-vertex group inclusion $\mathbb{G}_e \to \mathbb{G}_v$,
then the edge $e$ forms a loop.
In $T$ this means that if $S$ is obtained from $T$ by equivariantly collapsing an orbit of edges,
then there is no equivariant map $S \to T$.
We also assume that edge groups of $\mathbb{G}$ or equivalently edge stabilizers in $T$ are finitely generated.
If we begin with a graph of groups $\mathcal{G}$ which is not reduced,
we may obtain a reduced graph of groups by the operation of \emph{forest collapse,}
see \cite{Myself}.

\paragraph{Morphisms of graphs of groups.}
For a vertex $v$ of a graph of groups,
write $\st(v)$ for the set of oriented edges with initial vertex $v$.
Let $\mathcal{G}$ and $\mathcal{G}'$ be graphs of groups.
A \emph{morphism in the sense of Bass} of graphs of groups $f\colon \mathcal{G} \to \mathcal{G}'$
is a continuous map of underlying graphs taking vertices to vertices and edges to edges
together with homomorphisms $f_v\colon \mathcal{G}_v \to \mathcal{G}'_{f(v)}$
and $f_e\colon \mathcal{G}_e \to \mathcal{G}'_{f(e)}$
for each vertex $v$ and edge $e$ of $\mathcal{G}$
and for each oriented edge $e \in \st(v)$ of $\mathcal{G}$, an element $g_e \in \mathcal{G}'_{f(v)}$
with the property that for all $h \in \mathcal{G}_e$, we have
\[  f_v \iota_e (h)g_e = g_e\iota_{f(e)}f_e(h).  \]
We say a morphism is \emph{Stallings} if each $f_e$ and $f_v$ is injective.
In \cite{Myself}, a notion of a \emph{map} of graphs of groups $f\colon \mathcal{G} \to \mathcal{G}'$ is defined.
Up to homotopy and collapsing subgraphs in $\mathcal{G}$,
each map in that sense may be transformed into a morphism in the sense of Bass
by subdividing each edge of $\mathcal{G}$ into finitely many edges.

Given a vertex $v$ of a graph of groups $\mathcal{G}$,
the set of \emph{directions} $D_v$ based at $v$ is the disjoint union
\[  \coprod_{e \in \st(v)} \mathcal{G}_v/\iota_e(\mathcal{G}_e) \times \{e\}.  \]
Given a morphism $f\colon \mathcal{G} \to \mathcal{G}'$ in the sense of Bass
and a vertex $v$ of $\mathcal{G}$, there is an induced map on directions
$D_vf\colon D_v \to D_{f(v)}$ defined as
\[  D_vf(g\iota_e(\mathcal{G}_e),e) = (f_v(g)g_e\iota_{f(e)}(\mathcal{G}'_{f(e)}),f(e)). \]
The condition on $g_e$ from the previous paragraph implies that this map is well-defined.
A Stallings map $f\colon \mathcal{G} \to \mathcal{G}'$ is an \emph{immersion}
if each $D_vf$ is injective.

\paragraph{Turns.}
The obvious action of $\mathcal{G}_v$ on $D_v$ extends to a diagonal action 
on the set of unordered pairs of directions based at $v$.
A \emph{turn} based at $v$ is the $\mathcal{G}_v$-orbit of such an unordered pair.
A turn is \emph{degenerate} if it is represented by a pair of identical directions
and is otherwise \emph{nondegenerate.}
Given a morphism $f\colon \mathcal{G} \to \mathcal{G}'$ in the sense of Bass,
the map $D_vf$ extends to a map of turns $D_vf$ based at $v$
by the rule that $D_vf([\{d,d'\}]) = [\{D_vf(d),D_vf(d')\}]$.
We say that a turn based at $f(v)$ is \emph{taken by $\mathcal{G}$ at $v$}
if that turn is in the image of $D_vf$.
A foldable map is an immersion if and only if each $D_vf$
maps nondegenerate turns to nondegenerate turns.

\begin{ex}
    \label{immersionex}
    Consider the graph of groups $\mathbb{G}$ described as follows:
    the underlying graph of $\mathbb{G}$ is a tree with three edges, three leaves
    and one vertex of valence three.
    All vertex groups are cyclic of order two, and all edge groups are trivial.
    Color the edges of $\mathbb{G}$ red, orange and blue, see \Cref{immersionfig},
    and label them $r$, $o$ and $b$, respectively and orient them towards the center vertex.
    Consider another graph of groups $\mathcal{G}$.
    The underlying graph of $\mathcal{G}$ is also a finite tree with three leaves 
    and one vertex of valence three.
    In $\mathcal{G}$ the leaves have vertex group $C_2$ and all other groups are trivial.
    We describe an immersion $\mathcal{G} \to \mathbb{G}$
    by describing the images of the graph-of-groups edge paths from the leaves to the center vertex.
    They are $\bar r\star r\bar b\star b\bar r\star r\star$ and $o$ and $r$,
    where $\star$ denotes the nontrivial element of $C_2$, see \Cref{immersionfig},
    where in the figure a $\star$ indicates that the map $f$ 
    sends the nearest pair of edges in $\mathcal{G}$
    to a turn in $\mathbb{G}$ determined by the nontrivial element of $C_2$.
    It is easy to check that $f$ so described is an immersion.
\end{ex}

\begin{figure}[ht]
    \begin{center}
	    \def\svgwidth{\columnwidth}
	        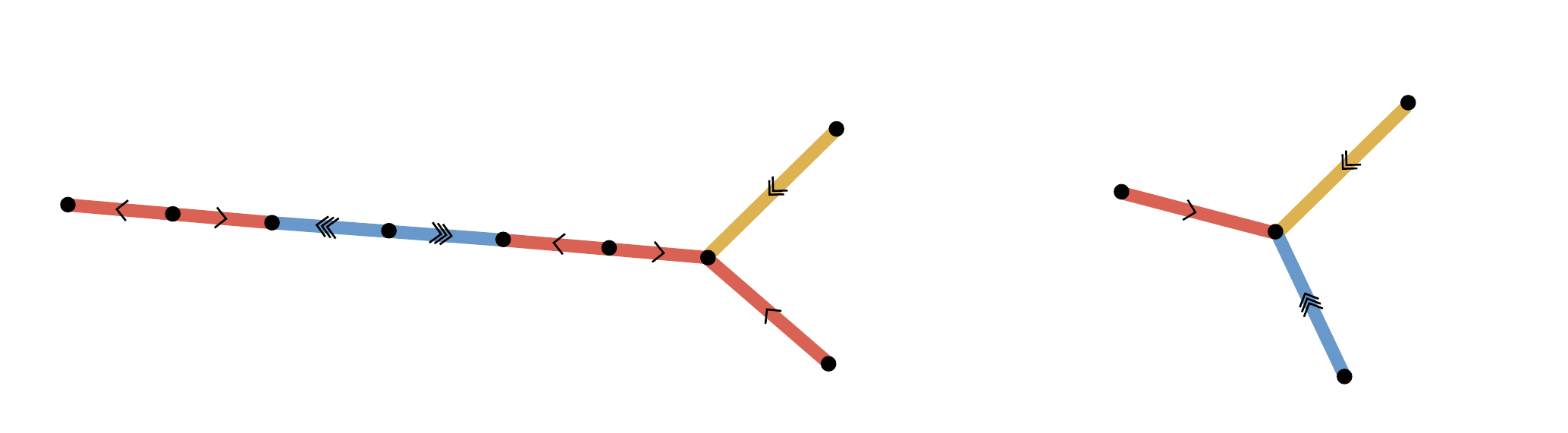
	
        \caption{Right: The graph of groups $\mathbb{G}$. Left: The immersion $f\colon \mathcal{G} \to \mathbb{G}$.}
        \label{immersionfig}
    \end{center}
\end{figure}

\paragraph{Stallings graphs of groups.}
Let $H$ be a subgroup of $G$,
and consider the minimal $H$-invariant subtree $T_H$ of $T$.
The inclusion $T_H \to T$ descends to an immersion in the sense of Bass \cite{Bass}
of quotient graphs of groups $s\colon H\doublebackslash T_H \to \mathbb{G} = G \doublebackslash T$.
The resulting map on fundamental groups $s_\sharp\colon \pi_1(H\doublebackslash T_H) \to \pi_1(\mathbb{G})$
has image conjugate to $H$.
If $H$ is convex-cocompact in $T$, then $H\doublebackslash T_H$ is a finite graph of groups.
The graph of groups $H\doublebackslash T_H$ is \emph{core} in the sense that
there is no proper subgraph of groups whose inclusion is a homotopy equivalence in the sense of \cite{Myself}.
If $\mathcal{G}$ is a finite, not necessarily connected, core graph of groups
equipped with an immersion $s\colon \mathcal{G} \to \mathbb{G}$,
we say that $\mathcal{G}$ is a \emph{Stallings graph of groups.}
If $\mathcal{G}_1,\ldots,\mathcal{G}_k$ are the components of $\mathcal{G}$,
we say that $\mathcal{G}$ \emph{represents}
the collection of conjugacy classes 
$\{[s_\sharp(\pi_1(\mathcal{G}_1))],\ldots,[s_\sharp(\pi_1(\mathcal{G}_k))]\}$.
If $g$ is an element of $G$, we may represent $g$ 
by a Stallings graph of groups for the cyclic subgroup $\langle g \rangle$.
Thus for each finite collection $C$ of conjugacy classes of convex-cocompact subgroups and elements of $G$,
there is a finite Stallings graph of groups $s\colon \mathcal{G} \to \mathbb{G}$ representing $C$.

\paragraph{The star graph.}
Let $w$ be a vertex of a graph of groups $\mathcal{G}$
equipped with a Stallings morphism $f\colon \mathcal{G} \to \mathbb{G}$.
The \emph{piece of $w$} is the following graph.
It has a center vertex $w$ and one edge with initial vertex $w$ for each element of $D_w$.
We think of the edge for each $d \in D_w$ as (temporarily) missing a terminal vertex.
The \emph{star graph of $\mathcal{G}$} is a bipartite graph defined as follows.
Begin with vertex set the disjoint union $\coprod_w D_w$ 
as $w$ varies over the vertex set of $\mathbb{G}$.
For each vertex $v$ of $\mathcal{G}$, let $w = f(v)$ and attach $|\mathbb{G}_w : f_v(\mathcal{G}_v)|$
copies of the piece of $v$ in the following way.
For $g \in \mathbb{G}_w$ representing a coset of $\mathbb{G} / f_v(\mathcal{G}_v)$, 
the $gf_v(\mathcal{G}_v)$-copy of the piece of $v$
is attached so that the edge of the piece of $v$ corresponding to the direction $d \in D_v$
has terminal vertex $g.D_vf(d) \in D_w$.
Observe that if $g' = gf_v(h)$, then
\[
    g'.D_vf(D_v) = gf_v(h).D_vf(D_v) = g.D_vf(h.D_v) = g.D_vf(D_v),
\]
so the isomorphism type of the star graph is well-defined independent of the choice
of coset representatives.

The star graph of $\mathcal{G}$ thus is the disjoint union of several graphs 
(which need not themselves be connected)
$\Gamma_w$, one for each vertex $w$ of $\mathbb{G}$,
containing the vertices corresponding to directions in $D_w$.
We say the star graph of $\mathcal{G}$ is \emph{disconnected}
if for some vertex $w$ of $\mathbb{G}$,
the graph $\Gamma_w$ is disconnected.
Supposing each $\Gamma_w$ is connected,
we say that the star graph has a \emph{cut vertex}
if some $\Gamma_w$ has a vertex $d$ corresponding to a direction in $D_w$,
the removal of which disconnects $\Gamma_w$.
The claim of \Cref{maintheorem}
is that if $\mathcal{G}$ is a Stallings graph of groups
representing a jointly simple collection of conjugacy classes of elements and convex-cocompact subgroups,
then the star graph of $\mathcal{G}$ is either disconnected or has a cut vertex.

Notice that if $f$ is an immersion, the star graph has no loops of length two.
Notice as well that a pair of distinct directions $d$ and $d' \in D_w$
are connected in $\Gamma_w$ by a path of length two if and only if the turn $[\{d,d'\}]$
is taken by $\mathcal{G}$.

Let us remark that in the case that $G$ is finite-rank nonabelian free
and $\mathcal{G}$ represents a finite collection of conjugacy classes of elements of $G$,
the star graph is typically thought of as not being bipartite:
since each $v \in \mathcal{G}$ has valence two and trivial vertex group,
each piece of $v$ may be thought of as an edge, rather than two edges joined at a common vertex.
Our treatment is inspired by \cite{Vogtmann,BestvinaFeighnHandel},
where the \emph{pieces} of the star graph are constucted by ``snipping''
midpoints of edges of $\mathcal{G}$.
Indeed, our construction is equivalent to doing this ``snipping'' in the Bass--Serre tree.

\begin{ex}
    Consider the immersion $f\colon \mathcal{G} \to \mathbb{G}$ from \Cref{immersionex}.
    The star graph for $\mathcal{G}$ has twelve vertices corresponding to directions in $\mathbb{G}$
    and fifteen vertices corresponding to the nine vertices of $\mathcal{G}$,
    six of which have trivial vertex group (and thus index two in $C_2$),
    see \Cref{stargraphfig}.
    Most pieces in the star graph are bivalent, either because they correspond to a valence-two vertex
    in $\mathcal{G}$ or because they correspond to a leaf with $C_2$ vertex group.
    The trivalent pieces connect both directions with underlying oriented edge $r$
    to one of the directions with underlying oriented edge $o$.
    Note that the star graph is connected in our sense,
    and that both of the vertices corresponding to directions with underlying oriented edge $r$
    are cut vertices.
\end{ex}

\begin{figure}[ht]
    \begin{center}
	    \def\svgwidth{\columnwidth}
	        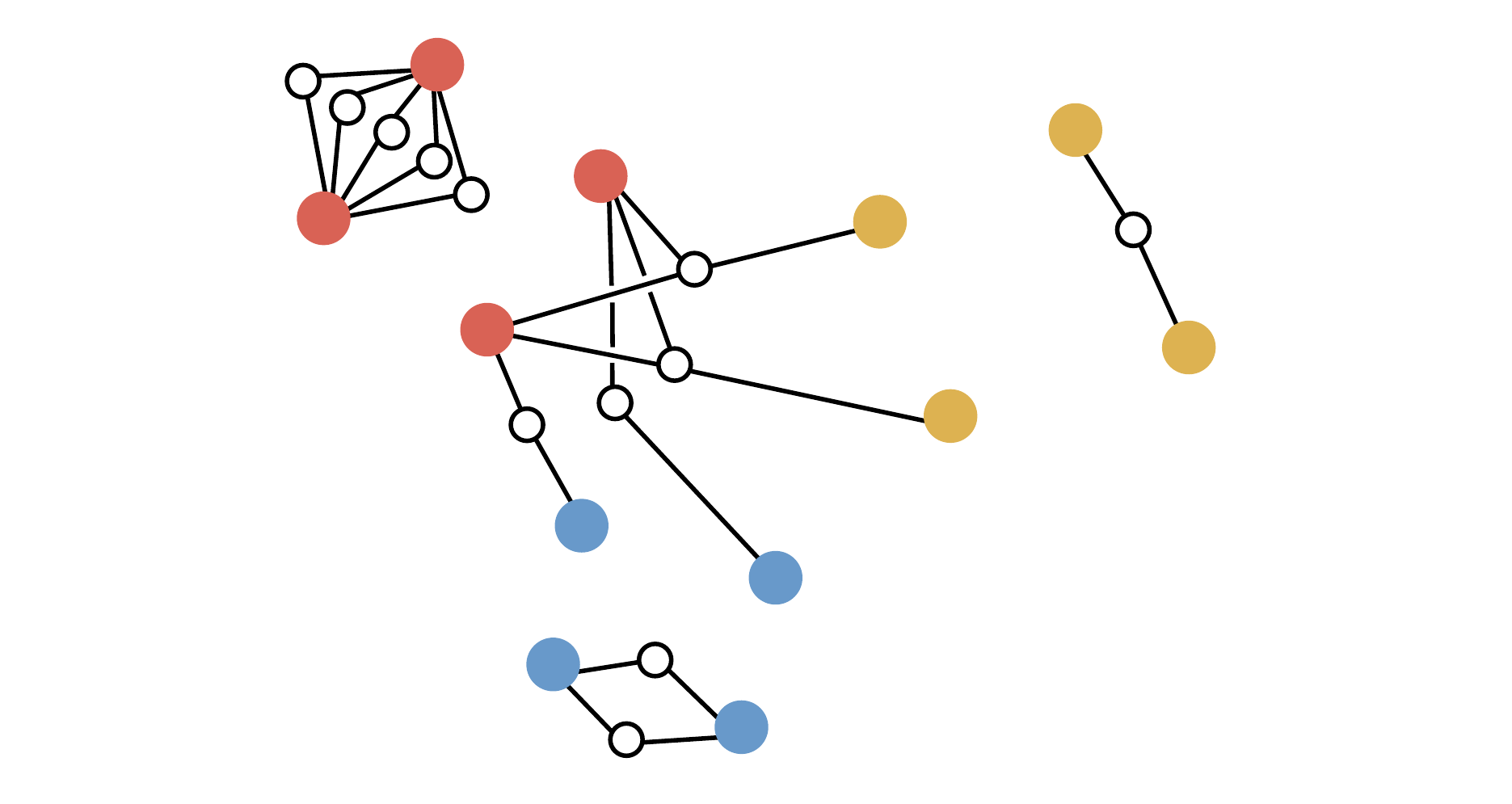
	
        \caption{The star graph of the immersion $f\colon \mathcal{G} \to \mathbb{G}$.}
        \label{stargraphfig}
    \end{center}
\end{figure}

\begin{lem}
    \label{mapsgivemaps}
    Suppose $f\colon \mathcal{G} \to \mathbb{G}$ and $f'\colon \mathcal{G}' \to \mathbb{G}$ 
    are Stallings morphisms of graphs of groups
    and that there exists a Stallings morphism $s\colon \mathcal{G} \to \mathcal{G}'$
    such that the following diagram commutes
    \[  \begin{tikzcd}
        \mathcal{G} \ar[rr, "s"] \ar[rd, "f"] & & \mathcal{G}' \ar[ld, "f'"'] \\
                                              & \mathbb{G}. &
    \end{tikzcd}\]
    Then each turn in $\mathbb{G}$ that is taken by $\mathcal{G}$ at a vertex $v$
    is taken by $\mathcal{G}'$ at the vertex $s(v)$.
\end{lem}

\begin{proof}
    The equation $f's = f$ implies that for each vertex $v$ of $\mathcal{G}$,
    we have $D_vf = D_{s(v)}f'D_vs$.
    It follows that if a pair of directions $d$ and $d'$ 
    are connected in a piece of $v$ in the star graph of $\mathcal{G}$
    then $d$ and $d'$ are connected in a piece of $s(v)$ in the star graph of $\mathcal{G}'$,
    and the lemma follows.
\end{proof}

\section{Almost \texorpdfstring{$\mathbb{G}$}{G} graphs of groups.}
Suppose $\mathcal{G}$ is a finite graph of groups
equipped with a Stallings morphism $f\colon \mathcal{G} \to \mathbb{G}$
which is a homotopy equivalence in the sense of \cite{Myself};
it follows that the induced map on fundamental groups $f_\sharp$ is an isomorphism.
By the main result of \cite{BestvinaFeighnGTrees},
the map $f$ may be represented as a finite product of (Stallings) \emph{folds} followed by an isomorphism.
In the Bass--Serre tree $S$ of $\mathcal{G}$, each fold $G$-equivariantly identifies a pair of edges.
We say that $\mathcal{G}$ is \emph{almost $\mathbb{G}$} if a single fold suffices.
Examining the cases in \cite{BestvinaFeighnGTrees}, we see that the fold is of ``Type I'' or ``Type II''
according to whether the edges involved in the fold are in identical or different orbits.
Folds of ``Type III'' are not possible because they change the elliptic subgroups in $\mathcal{G}$,
contradicting the assumption that the map $f\colon \mathcal{G} \to \mathbb{G}$ is a homotopy equivalence.

\begin{prop}
    \label{cutvertex}
    Suppose $f\colon \mathcal{G} \to \mathbb{G}$ is an almost $\mathbb{G}$ graph of groups.
    The star graph of $\mathcal{G}$ has a cut vertex.
\end{prop}

\begin{proof}
    Suppose first that the fold is of ``Type I,'' 
    i.e.~that there are distinct edges $e$ and $e'$ of $\mathcal{G}$
    sharing a common initial vertex $x$ satisfying $f(e) = f(e')$.
    Let $u$ and $v$ be the (distinct) terminal vertices of $e$ and $e'$ in $\mathcal{G}$
    and $w = f(u) = f(v)$ in $\mathbb{G}$.
    We claim that the star graph of $\mathcal{G}$ is connected,
    and that each direction with underlying oriented edge $f(e) = f(e')$
    is a cut vertex of the star graph.

    Notice first that if $y$ is a vertex of $\mathbb{G}$ not equal to $f(x)$ or $w$,
    then $\Gamma_y$ is the cone on $D_y$,
    because after folding $e$ and $e'$, the map $f$ becomes an isomorphism.
    For the same reason, if $y = f(x) \ne w$,
    the graph $\Gamma_y$ is obtained from the cone on $D_y$
    by attaching an extra edge from the cone point to each direction $d$ 
    with underlying oriented edge $f(\bar e) = f(\bar e')$.
    Thus these graphs are connected.
    After folding, the graph $\Gamma_w$ must be the cone on $D_w$,
    which is only possible if each direction $d \in D_w$ is in the image of
    either some $g.D_uf$ or $g.D_vf$.
    The direction $d$ is thus connected to a direction with underlying oriented edge $f(e) = f(e')$.
    We claim each pair of directions $d$ and $d'$ with underlying oriented edge $f(e) = f(e')$
    is connected in the star graph, which shows that $\Gamma_w$ is connected.
    Indeed, $d = g.d'$ for some $g \in \mathbb{G}_w$.
    Since $f$ folds to an isomorphism, we have that 
    $\mathbb{G}_w = \langle f_u(\mathcal{G}_u)\cup f_v(\mathcal{G}_v)\rangle$,
    so we may express $g$ as a word in elements of $f_u(\mathcal{G}_u) \cup f_v(\mathcal{G}_v)$,
    and an easy induction argument on the length of the word shows that $d$ and $d'$ are connected.

    It remains to see that each direction $d \in D_w$ 
    with underlying oriented edge $f(e) = f(e')$ is a cut vertex.
    Notice that because after folding, the map $f$ becomes an isomorphism,
    $e$ and $e'$ are the only edges incident to $u$ and $v$ respectively
    mapped to the same edge by $f$.
    It follows that directions in $D_w$ with underlying oriented edge not equal to $f(e) = f(e')$
    are in the image of some translates of $D_uf$ or $D_vf$ but not both.
    Put another way, since $\Gamma_w$ is connected, we have
    \[
        D_w = \left(\bigcup_{g\in \mathbb{G}_w}g.D_uf(D_u)\right)\cup
        \left(\bigcup_{g\in \mathbb{G}_w}g.D_vf(D_v)\right)
    \] and \[  
        \mathbb{G}_w/\iota_{f(e)}(\mathbb{G}_{f(e)})\times \{f(e)\} 
        = \left(\bigcup_{g \in \mathbb{G}_w}g.D_uf(D_u)\right) \cap 
        \left(\bigcup_{g \in \mathbb{G}_w}g.D_vf(D_v)\right).
        \]
    This shows that each $d$ with underlying oriented edge $f(e) = f(e')$ is a cut vertex
    in $\Gamma_w$ unless there is a nondegenerate turn $[\{d,d'\}]$ taken by $\mathcal{G}$ at both $u$ and $v$.
    Consider the conjugacy class of an element of $G$ whose representative in $\mathcal{G}$
    is a loop contained in $e \cup e'$ making turns at $u$ and $v$ mapping to $[\{d,d'\}]$,
    and making turns at $x$ whose $D_xf$-images are degenerate,
    for example the loop
    \[
        h'\bar e' eh\bar e e'
    \]
    for appropriate choices of $h \in \mathcal{G}_u$ and $h' \in \mathcal{G}_v$.
    By construction, this conjugacy class is hyperbolic in $\mathcal{G}$ but elliptic in $\mathbb{G}$,
    a contradiction.
    
    Now suppose the fold is of ``Type II,''
    i.e.~that there is an edge $e$ of $\mathcal{G}$ with initial vertex $u$ and terminal vertex $v$
    and an element $g \in \mathcal{G}_u$ 
    such that in the Bass--Serre tree $S$ of $\mathcal{G}$, the fold $G$-equivariantly identifiies
    a lift $\tilde e$ of $e$ and $g.\tilde e$,
    where $g$ is identified with an element of the stabilizer of the corresponding lift of $u$.
    Thus the map $f\colon \mathcal{G} \to \mathbb{G}$ is an isomorphism of underlying graphs,
    while the maps $f_e\colon \mathcal{G}_e \to \mathbb{G}_{f(e)}$ 
    and $f_v\colon \mathcal{G}_v \to \mathbb{G}_{f(v)}$
    are injective but not surjective.
    If, abusing notation, we identify $g \in G$ 
    with the corresponding elements of $\mathbb{G}_{f(v)}$ and $\mathbb{G}_{f(e)}$,
    we have $\mathbb{G}_{f(v)} = \langle f_v(\mathcal{G}_v)\cup \{g\}\rangle$
    and $\mathbb{G}_{f(e)} = \langle f_e(\mathcal{G}_e) \cup \{g\}\rangle$.
    Furthermore, we have $\iota_e(\mathcal{G}_e) = \mathcal{G}_v$,
    and therefore $\iota_{f(e)}(\mathbb{G}_{f(e)}) = \mathbb{G}_{f(v)}$.
    Indeed supposing not, consider the conjugacy class determined by the loop
    \[  geg'\bar e \]
    for some $g' \in \mathcal{G}_v\setminus \mathcal{G}_e$.
    Then this conjugacy class is hyperbolic in $\mathcal{G}$ but elliptic in $\mathbb{G}$,
    a contradiction.
    It follows that the edge $f(e)$ forms a loop in $\mathbb{G}$
    and therefore there exists a direction $d$ with underlying oriented edge different from $f(e)$.
    It is clear that $\iota_{\bar e}(\mathcal{G}_e) \ne \mathcal{G}_v$,
    for there would be nothing to fold if not.
    Consider the fundamental group of the subgraph of groups corresponding to the edge $f(e)$.
    It is convex-cocompact in $\mathbb{G}$ and $\mathcal{G}$,
    and if we had $\iota_{f(\bar e)}(\mathbb{G}_{f(e)}) = \mathbb{G}_{f(v)}$,
    its minimal subtree in the Bass--Serre tree of $\mathbb{G}$ would be a line,
    while it is not a line in the Bass--Serre tree of $\mathcal{G}$.
    This contradicts the fact that the homotopy equivalence $f$ 
    restricts to an equivariant quasi-isometry of minimal subtrees.
    Therefore $\iota_{f(\bar e)}(\mathbb{G}_{f(e)}) \ne \mathbb{G}_{f((v)}$.
    Since after folding, the map $f$ is an isomorphism,
    and the fold does not alter vertex and edge groups other than $\mathcal{G}_v$ and $\mathcal{G}_e$,
    if $e'$ is an oriented edge with initial vertex $v$ in $\mathcal{G}$ distinct from $e$ and $\bar e$,
    we have $\mathbb{G}_{e'} = f_{e'}(\mathcal{G}_{e'})$ and
    $\iota_{f(e')}(\mathbb{G}_{f(e')}) \ne \mathbb{G}_{f(v)}$.
    If $g_{\bar e}$ is the element of $\mathbb{G}_{f(v)}$ as part of the definition of $f$,
    thinking of $g$ as an element of $\mathbb{G}_{f(e)}$,
    we have $g_{\bar e}\iota_{\bar e}(g)g_{\bar e}^{-1} \in f_v(\mathcal{G}_v)$ by construction.

    By the argument in the case of a ``Type I'' fold, each $\Gamma_y$ is connected 
    for $y \ne f(v)$ a vertex of $\mathbb{G}$.
    Since after folding, the graph $\Gamma_{f(v)}$ is the cone on $D_{f(v)}$,
    we have that each direction $d$ in $D_{f(v)}$ is in the image of some $x.D_vf$
    and is thus connected to the unique direction $d'$ with underling oriented edge $f(e)$.

    We claim that the unique direction with underlying oriented edge $f(e)$ is a cut vertex.
    This is equivalent to the claim that the pieces of $v$ in $\Gamma_{f(v)}$ overlap
    exactly in the unique direction with underlying oriented edge $f(e)$.
    So consider a direction $d$ in $D_{f(v)}$ with underlying oriented edge not equal to $f(e)$,
    and suppose $d = x.D_vf(c) = x'D_vf(c')$ for directions $c$ and $c'$ in $D_v$
    and elements $x$ and $x'$ in $\mathbb{G}_{f(v)}$.
    Write $c = (y\iota_{e'}(\mathcal{G}_{e'}),e')$ and $c' = (y'\iota_{e'}(\mathcal{G}_{e'}),e')$
    and observe that we therefore have
    \[  xf_v(y)g_{e'}\iota_{f(e')}(\mathbb{G}_{f(e')}) 
    = x'f_v(y')g_{e'}\iota_{f(e')}(\mathbb{G}_{f(e')}). \]
    If the oriented edge $e'$ satisfies $e' \ne \bar e$,
    then $\mathbb{G}_{f(e')} = f_{e'}(\mathcal{G}_{e'})$,
    and since we have 
    \[
        g_{e'}\iota_{f(e')}f_{e'}(\mathcal{G}_{e'}) g_{e'}^{-1} = f_v\iota_{e'}(\mathcal{G}_{e'}),
    \]
    we conclude that $x$ and $x'$ determine the same coset of $\mathbb{G}_{f(v)}/f_v(\mathcal{G}_v)$.
    In the case $e' = \bar e$,
    we do not have $\mathbb{G}_{f(\bar e)} = f_{\bar e}(\mathcal{G}_{e})$,
    but it is nevertheless true that 
    $g_{\bar e}\iota_{f(\bar e)}(\mathbb{G}_{f(e)})g_{\bar e}^{-1} \le f_v(\mathcal{G}_v)$
    and we again conclude that $x$ and $x'$ determine the same coset of 
    $\mathbb{G}_{f(v)}/f_v(\mathcal{G}_v)$,
    proving the claim.
\end{proof}

\begin{ex}
    Consider the standard graph of groups presentation for $BS(1,6)$
    with one vertex $v$ and one loop edge $e$.
    We have $\mathbb{G}_v = \mathbb{Z}$ and $\mathbb{G}_e = \mathbb{Z}$;
    the inclusion $\iota_e$ is an isomorphism which we will treat as the identity,
    and the inclusion $\iota_{\bar e}$ is multiplication by $6$.
    An example of an almost $\mathbb{G}$ graph of groups is the following graph of groups $\mathcal{G}$.
    As a graph of groups $\mathcal{G}$ is abstractly isomorphic to $\mathbb{G}$,
    the graph map $f$ is an isomorphism, $g_e$ and $g_{\bar e}$ are trivial,
    but the maps $f_v$ and $f_e$ are multiplication by $3$.
    We have
    \[  D_v = \{e, (6\mathbb{Z},\bar e), (1+ 6\mathbb{Z},\bar e), (2+6\mathbb{Z},\bar e),
    (3 + 6\mathbb{Z},\bar e), (4+6\mathbb{Z},\bar e), (5 + 6\mathbb{Z}, \bar e)\}. \]
    The map $D_vf$ sends $e$ to $e$ and $(k + 6\mathbb{Z},\bar e)$ to $(3k + 6\mathbb{Z},\bar e)$.
    As $g$ varies over the coset representatives of $\mathbb{G}_v/f_v(\mathcal{G}_v)$,
    pairs of directions with underlying oriented edge $\bar e$ are connected by a path of length two
    in the star graph if and only if they differ by $3\mathbb{Z}$.
    The direction $e$ is a cut vertex.
\end{ex}

\begin{proof}[Proof of \Cref{maintheorem}]
    Let $C$ be a collection of jointly simple conjugacy classes 
    of elements and convex-cocompact subgroups of $G$.
    Choose a tree $S$ in the reduced deformation space of $T$
    which collapses onto the tree $S'$ in which each conjugacy class in $C$ is elliptic.
    By assumption, there is a $G$-equivariant map $\tilde f\colon S \to T$,
    which after $G$-equivariantly subdividing edges of $S$ we may assume is simplicial
    (edges to edges, vertices to vertices).
    Having performed this subdivision, consider the quotient graph of groups $G\doublebackslash S$.
    The map $\tilde f$ descends 
    to a Stallings morphism in the sense of Bass $f\colon \mathcal{G} \to \mathbb{G}$,
    where $\mathbb{G} = G\doublebackslash T$.
    The preimages of vertices of $G\doublebackslash S'$ in $\mathcal{G}$ determine one
    (if $G\doublebackslash S'$ is a loop) or two (if $G\doublebackslash S'$ is not a loop)
    subgraphs of groups.
    We will suppose there are two, call them $\mathcal{H}_1$ and $\mathcal{H}_2$;
    the case of one is simpler and only superficially different.
    These subgraphs of groups have the property that they are disjoint in $\mathcal{G}$ 
    and that each conjugacy class in $C$
    is conjugate either into $\pi_1(\mathcal{H}_1)$ or $\pi_1(\mathcal{H}_2)$.
    The map $f\colon \mathcal{G} \to \mathbb{G}$ may be written as a product of folds.
    By first performing all folds that involve only edges of $\mathcal{H}_1$ 
    or only edges of $\mathcal{H}_2$,
    we may assume that $f$ restricts to an immersion on $\mathcal{H}_1$ and $\mathcal{H}_2$.
    Since each conjugacy class in $C$ is conjugate into
    either $\pi_1(\mathcal{H}_1)$ or $\pi_1(\mathcal{H}_2)$,
    it follows that the Stallings graph of groups $s\colon \mathcal{G}_C\to \mathbb{G}$ representing $C$ 
    may be factored through $f$.

    We distinguish two cases.
    In the first case, $f$ is an isomorphism,
    so there is a (necessarily separating) edge $e$
    which is the complement in $\mathbb{G}$ of $f(\mathcal{H}_1)$ and $f(\mathcal{H}_2)$.
    The edge $e$ is therefore not in the image of $s$
    and any direction with underlying oriented edge $e$ or $\bar e$
    is isolated in the star graph of $C$.
    In this case the star graph is disconnected.

    In the second case, $f$ is not an isomorphism,
    so there is a further sequence of folds we may perform.
    It follows that there is a map $\mathcal{G} \to \mathcal{G}'$,
    where $\mathcal{G}'$ is an almost $\mathbb{G}$ graph of groups.
    By \Cref{cutvertex}, the star graph of $\mathcal{G}'$ has a cut vertex.
    By the contrapositive of \Cref{mapsgivemaps}, every turn not taken by $\mathcal{G}'$
    is not taken by $\mathcal{G}$ nor by $\mathcal{G}_C$.
    It follows that if the star graph of $C$ is connected,
    any cut vertex for the star graph of $\mathcal{G}'$ will be a cut vertex for the star graph of $C$.
    Therefore either the star graph of $C$ is disconnected or it has a cut vertex.
\end{proof}

%% file: figures/immersion.pdf_tex
\begingroup%
  \makeatletter%
  \providecommand\color[2][]{%
    \errmessage{(Inkscape) Color is used for the text in Inkscape, but the package 'color.sty' is not loaded}%
    \renewcommand\color[2][]{}%
  }%
  \providecommand\transparent[1]{%
    \errmessage{(Inkscape) Transparency is used (non-zero) for the text in Inkscape, but the package 'transparent.sty' is not loaded}%
    \renewcommand\transparent[1]{}%
  }%
  \providecommand\rotatebox[2]{#2}%
  \newcommand*\fsize{\dimexpr\f@size pt\relax}%
  \newcommand*\lineheight[1]{\fontsize{\fsize}{#1\fsize}\selectfont}%
  \ifx\svgwidth\undefined%
    \setlength{\unitlength}{595.27559055bp}%
    \ifx\svgscale\undefined%
      \relax%
    \else%
      \setlength{\unitlength}{\unitlength * \real{\svgscale}}%
    \fi%
  \else%
    \setlength{\unitlength}{\svgwidth}%
  \fi%
  \global\let\svgwidth\undefined%
  \global\let\svgscale\undefined%
  \makeatother%
  \begin{picture}(1,0.27619048)%
    \lineheight{1}%
    \setlength\tabcolsep{0pt}%
    \put(0,0){\includegraphics[width=\unitlength,page=1]{immersion.pdf}}%
    \put(0.69963045,0.16431159){\color[rgb]{0,0,0}\makebox(0,0)[lt]{\lineheight{1.25}\smash{\begin{tabular}[t]{l}$C_2$\end{tabular}}}}%
    \put(0.79144098,0.14347476){\color[rgb]{0,0,0}\makebox(0,0)[lt]{\lineheight{1.25}\smash{\begin{tabular}[t]{l}$C_2$\end{tabular}}}}%
    \put(0.86041303,0.01214224){\color[rgb]{0,0,0}\makebox(0,0)[lt]{\lineheight{1.25}\smash{\begin{tabular}[t]{l}$C_2$\end{tabular}}}}%
    \put(0.89218772,0.22041199){\color[rgb]{0,0,0}\makebox(0,0)[lt]{\lineheight{1.25}\smash{\begin{tabular}[t]{l}$C_2$\end{tabular}}}}%
    \put(0.53364832,0.02679258){\color[rgb]{0,0,0}\makebox(0,0)[lt]{\lineheight{1.25}\smash{\begin{tabular}[t]{l}$C_2$\end{tabular}}}}%
    \put(0.53055929,0.20435135){\color[rgb]{0,0,0}\makebox(0,0)[lt]{\lineheight{1.25}\smash{\begin{tabular}[t]{l}$C_2$\end{tabular}}}}%
    \put(0.02185742,0.16160195){\color[rgb]{0,0,0}\makebox(0,0)[lt]{\lineheight{1.25}\smash{\begin{tabular}[t]{l}$C_2$\end{tabular}}}}%
    \put(0.43715268,0.09395398){\color[rgb]{0,0,0}\makebox(0,0)[lt]{\lineheight{1.25}\smash{\begin{tabular}[t]{l}$\star$\end{tabular}}}}%
    \put(0.43644424,0.1210992){\color[rgb]{0,0,0}\makebox(0,0)[lt]{\lineheight{1.25}\smash{\begin{tabular}[t]{l}$\star$\end{tabular}}}}%
    \put(0.38289392,0.09689433){\color[rgb]{0,0,0}\makebox(0,0)[lt]{\lineheight{1.25}\smash{\begin{tabular}[t]{l}$\star$\end{tabular}}}}%
    \put(0.24199987,0.10713884){\color[rgb]{0,0,0}\makebox(0,0)[lt]{\lineheight{1.25}\smash{\begin{tabular}[t]{l}$\star$\end{tabular}}}}%
    \put(0.10469813,0.1187963){\color[rgb]{0,0,0}\makebox(0,0)[lt]{\lineheight{1.25}\smash{\begin{tabular}[t]{l}$\star$\end{tabular}}}}%
  \end{picture}%
\endgroup%

%% file: figures/stargraph.pdf_tex
\begingroup%
  \makeatletter%
  \providecommand\color[2][]{%
    \errmessage{(Inkscape) Color is used for the text in Inkscape, but the package 'color.sty' is not loaded}%
    \renewcommand\color[2][]{}%
  }%
  \providecommand\transparent[1]{%
    \errmessage{(Inkscape) Transparency is used (non-zero) for the text in Inkscape, but the package 'transparent.sty' is not loaded}%
    \renewcommand\transparent[1]{}%
  }%
  \providecommand\rotatebox[2]{#2}%
  \newcommand*\fsize{\dimexpr\f@size pt\relax}%
  \newcommand*\lineheight[1]{\fontsize{\fsize}{#1\fsize}\selectfont}%
  \ifx\svgwidth\undefined%
    \setlength{\unitlength}{538.58267717bp}%
    \ifx\svgscale\undefined%
      \relax%
    \else%
      \setlength{\unitlength}{\unitlength * \real{\svgscale}}%
    \fi%
  \else%
    \setlength{\unitlength}{\svgwidth}%
  \fi%
  \global\let\svgwidth\undefined%
  \global\let\svgscale\undefined%
  \makeatother%
  \begin{picture}(1,0.53157895)%
    \lineheight{1}%
    \setlength\tabcolsep{0pt}%
    \put(0,0){\includegraphics[width=\unitlength,page=1]{stargraph.pdf}}%
  \end{picture}%
\endgroup%